\documentclass[11pt,reqno]{amsart}

\usepackage{amsmath,amsopn,amssymb,amsthm,multicol}

\usepackage{hyperref}

\usepackage{graphicx}
\newcommand{\fr}{\mathfrak}

\newcommand{\op}{\operatorname}

 \newtheorem{lemma} {Lemma} [section]
\newtheorem{theorem}[lemma]{Theorem} 
\newtheorem{remark}[lemma] {Remark} 
\newtheorem{prop} [lemma]{Proposition}  
 
\newtheorem{corol}[lemma] {Corollary}

\numberwithin{equation}{section}

\begin{document}

\title{On the transitivity of Lie ideals and a characterization of perfect Lie algebras}
\author{Nikolaos Panagiotis Souris}
\address{University of Patras, Department of Mathematics, University Campus, 26504, Rio Patras, Greece}
\email{nsouris@upatras.gr}

\maketitle
\medskip

\begin{abstract}
We explore general intrinsic and extrinsic conditions that allow the transitivity of the relation of being a Lie ideal, in the sense that if a Lie algebra $\fr{h}$ is a subideal of a Lie algebra $\fr{g}$ (i.e. there exist Lie subalgebras $\fr{l}_0,\fr{l}_1,\dots,\fr{l}_n$ of $\fr{g}$ with $\fr{h}=\fr{l}_0\unlhd \fr{l}_1 \unlhd\cdots \unlhd \fr{l}_n=\fr{g}$), then $\fr{h}$ is an ideal of $\fr{g}$. We also prove that perfect Lie algebras of arbitrary dimension and over any field are intrinsically characterized by transitivity of this type; In particular, we show that a Lie algebra $\fr{h}$ is perfect (i.e. $\fr{h}=[\fr{h}, \fr{h}]$) if and only if for any Lie algebra $\fr{g}$ such that $\fr{h}$ is a subideal of $\fr{g}$, it follows that $\fr{h}$ is an ideal of $\fr{g}$.



\medskip
 \end{abstract}

\section{Introduction}

Although the normality relation in groups is not transitive in general, there exist structural conditions that allow the transitivity of normality in both finite and infinite groups (c.f. \cite{As-He}, \cite{Ro}, \cite{Ru} and references therein). For Lie algebras, an analogous question concerns the transitivity of the relation of being an ideal. A Lie algebra $\fr{h}$ is called a \emph{subideal} of a Lie algebra $\fr{g}$ (or a \emph{subinvariant} subalgebra of $\fr{g}$) if there exist finitely many Lie subalgebras $\fr{l}_0,\fr{l}_1,\dots,\fr{l}_n$ of $\fr{g}$ such that

 \begin{equation}\label{subidealrel}\fr{h}=\fr{l}_0\unlhd \fr{l}_1 \unlhd\cdots \unlhd \fr{l}_n=\fr{g},\end{equation}

\noindent where $\unlhd$ denotes the relation of being an ideal. Schenkman developed in \cite{Sh} a theory of finite-dimensional Lie subideals, and a recent generalization to Leibniz algebras is given in \cite{MiSiYu}. Subideals of Lie superalgebras have been studied in \cite{Sici}.  This article explores some general conditions that allow the transitivity of relation \eqref{subidealrel}, in the sense that if $\fr{h}$ is a subideal of a Lie algebra $\fr{g}$ then $\fr{h}$ is an ideal of $\fr{g}$.  We regard such transitivity conditions through the following questions.\\

\textbf{I}. Intrinsic conditions for $\fr{g}$: Conditions that depend only on the structure of $\fr{g}$, such that every subideal $\fr{h}$ of $\fr{g}$ is an ideal.\\

\textbf{II}. Intrinsic conditions for $\fr{h}$: Conditions that depend only on the structure of $\fr{h}$, such that whenever $\fr{h}$ is embedded as a subideal of some Lie algebra $\fr{g}$, then $\fr{h}$ is an ideal of $\fr{g}$.\\  

\textbf{III}. Extrinsic conditions for $\fr{h}$: Conditions that depend on the embedding of $\fr{h}$ as a subalgebra of $\fr{g}$, such that if $\fr{h}$ is a subideal of $\fr{g}$ then $\fr{h}$ is an ideal of $\fr{g}$.\\

In \cite{Ge-Mu}, the authors address question \textbf{I}. through the notion of a \emph{$\fr{t}$-algebra}, namely a Lie algebra $\fr{g}$ such that any subideal of $\fr{g}$ is an ideal of $\fr{g}$. They show that any $\fr{t}$-algebra over a field of characteristic zero is the direct sum $\fr{s}\oplus \fr{r}$, where $\fr{s}$ is a semisimple ideal and $\fr{r}$ is an abelian or almost abelian Lie algebra (see also \cite{Av} for the structure of almost abelian Lie algebras). Some related classes to $\fr{t}$-algebras are studied in \cite{Va}. In \cite{Kurda}, question \textbf{I}. is addressed in the context of Leibniz algebras.

For question \textbf{II}, recall that a Lie algebra $\fr{h}$ is called \emph{perfect} if it equals its commutator ideal $[\fr{h},\fr{h}]$, i.e. the ideal generated by elements of the form $[X,Y]$ for $X,Y\in \fr{h}$. Being a natural generalization of semisimple Lie algebras (at least in characteristic zero), but lacking such an explicitly understood structure, perfect Lie algebras are actively studied in terms of their algebraic and structural properties (\cite{BenS}, \cite{BuDeMo}, \cite{SuZhu}), cohomology (\cite{BuWa}), generalized classes (\cite{Pir}), representations (\cite{Bara}) as well as their interactions with mathematical physics (\cite{Sal}, \cite{Stu}). We answer question \textbf{II} in Section \ref{SecPer}, showing that perfect Lie algebras are precisely those Lie algebras $\fr{h}$ such that whenever $\fr{h}$ is embedded as a subideal of some Lie algebra $\fr{g}$, then $\fr{h}$ is an ideal of $\fr{g}$; If we make no distinction between a Lie algebra $\fr{h}$ and its isomorphic image under an injective Lie homomorphism, our main result, Theorem \ref{mainperfect}, is stated as follows.

\begin{theorem}\label{mainperfect2} Let $\fr{h}$ be a Lie algebra over a field $\mathbb K$.  The following are equivalent:\\
\noindent \emph{(i)} The Lie algebra $\fr{h}$ is perfect, i.e. $\fr{h}=[\fr{h},\fr{h}]$.\\ 
\noindent \emph{(ii)} For all Lie algebras $\fr{k}$, $\fr{g}$ over $\mathbb K$ such that $\fr{h}\unlhd\fr{k}\unlhd \fr{g}$, it follows that $\fr{h}\unlhd\fr{g}$. \end{theorem}

 The fact that (i) implies (ii) is known for finite dimensions (see e.g. \cite{MiSiYu}) but essentially follows from the Jacobi identity irrespectively of the dimension.  We prove the converse implication by firstly establishing in Lemma \ref{characteristic} that if an ideal of a Lie algebra satisfies the transitivity condition (ii) then it is \emph{characteristic} (i.e. invariant by derivations of the ambient algebra), and then for any non-perfect Lie algebra we construct a natural extension that does not involve the algebra as a characteristic ideal. We also prove a weaker version of the transitivity property of Theorem \ref{mainperfect2} for \emph{complete} Lie algebras, namely Lie algebras with trivial center and only inner derivations (Proposition \ref{complete}). We stress that the aforementioned results are valid regardless of the dimension, finite or infinite. As an application, by proving a generalized version of the main result of Y. Su and L. Zhu in \cite{SuZhu} (Theorem \ref{derived}), we show how the completeness of derivation algebras of centerless perfect Lie algebras can be essentially recovered from the transitivity condition (ii) (Remark \ref{SuShu}).

In Section \ref{SecCom}, we address question \textbf{III.} for finite-dimensional Lie algebras in characteristic zero. We show that the property for a subideal to be an ideal depends on the embedding of the \emph{radical} (i.e. the maximal solvable ideal) of the subideal in the ambient algebra (Proposition \ref{Leviconsequence}). Then we focus on subalgebras of real Lie algebras, defined by the skew-symmetry of their adjoint operators with respect to symmetric bilinear forms (Theorem \ref{picture}). Our prototype example is the wide class of \emph{compactly embedded} subalgebras (\cite{HilHof}, \cite{Wu}), which includes the Lie algebras of compact Lie subgroups and the isotropy algebras of Riemannian homogeneous manifolds. We prove that any compactly embedded subideal is an ideal (Corollary \ref{CorollaryCompactly}).  We also prove that any subideal containing an eigenspace of a Cartan involution in a real semisimple Lie algebra is an ideal (Corollary \ref{CartanSemi}).

Apart from the aspect of transitivity itself, questions \textbf{I.} - \textbf{III.} are naturally connected to the existence of self-normalizing subalgebras of Lie algebras. The \emph{normalizer} of $\fr{h}$ in $\fr{g}$ is the maximal subalgebra $N_{\fr{g}}(\fr{h})$ of $\fr{g}$ containing $\fr{h}$ as an ideal, while $\fr{h}$ is called \emph{self-normalizing} in $\fr{g}$ if $\fr{h}=N_{\fr{g}}(\fr{h})$.  The maximality of the normalizer allows subalgebras that satisfy general transitivity conditions to possess self-normalizing normalizers. A simple application of our results in Section \ref{GeodesicOrbit} yields the existence of large classes of self-normalizing subalgebras of Lie algebras (Theorem \ref{SelfNormalizingTh}).  In addition to their algebraic implications (e.g. \cite{Si}), self-normalizing subalgebras have interactions with geometry; For example, they appear as components of isometry algebras of some special homogeneous manifolds, the geodesic orbit manifolds (c.f. Corollary 1.12 in \cite{Sou}), the classification of which is a longstanding open problem in Riemannian geometry (we refer the interested reader to the recent monograph \cite{BerNik} on the subject). In this regard, we stress that part (v) of Theorem \ref{SelfNormalizingTh}, combined with results in \cite{Sou}, can be applied for the computation of isometry algebras of geodesic orbit metrics in compact simple Lie groups.

\section{The transitivity characterization of perfect Lie algebras}\label{SecPer}
   
The Lie algebras in this section are defined over an arbitrary field and have no dimensional restrictions. The initial aim of this section is to prove the following characterization of perfect Lie algebras. 

\begin{theorem}\label{mainperfect} Let $\fr{h}$ be a Lie algebra over a field $\mathbb K$. The following are equivalent.\\
\noindent \emph{(i)} The Lie algebra $\fr{h}$ is perfect.\\ 
\noindent \emph{(ii)} For any Lie algebra  $\fr{g}$ over $\mathbb K$ and for any injective Lie algebra homomorphism $i:\fr{h}\rightarrow \fr{g}$  such that $i(\fr{h})$ is a subideal of $\fr{g}$, it follows that $i(\fr{h})$ is an ideal of $\fr{g}$.\end{theorem}

\begin{remark} If $\fr{h}$ is a Lie algebra satisfying condition (ii) of Theorem \ref{mainperfect}, then any Lie algebra isomorphic to $\fr{h}$ also satisfies this condition.  Thus we will henceforth make no distinction between a Lie algebra $\fr{h}$ and its isomorphic image $i(\fr{h})$ under an injective Lie homomorphism $i:\fr{h}\rightarrow \fr{g}$, when it is clear from the context. Moreover, a simple inductive argument shows that a Lie algebra $\fr{h}$ satisfies condition (ii) of Theorem \ref{mainperfect} if and only if for any pair of Lie algebras $\fr{k}, \fr{g}$ such that $\fr{h}\unlhd \fr{k}\unlhd \fr{g}$, it follows that $\fr{h}\unlhd \fr{g}$. Therefore, theorems \ref{mainperfect} and \ref{mainperfect2} have equivalent statements.\end{remark}

Recall that a \emph{derivation} of a Lie algebra $\fr{h}$ is a linear endomorphism $f$ of $\fr{h}$ satisfying $f([X,Y])=[f(X),Y]+[X,f(Y)]$ for all $X,Y\in \fr{h}$.  Let $(D(\fr{h}), [\ , \ ]_D)$ denote the Lie algebra of derivations of $\fr{h}$.  The \emph{holomorph} of $\fr{h}$ is defined as the semidirect product

\[ H(\fr{h}):=\fr{h}\rtimes D(\fr{h}),\]

\noindent equipped with the Lie bracket

\begin{equation*}\label{brackethol}\big[(X,f),(Y,g)\big]_{H}:=\big([X,Y]+f(Y)-g(X),\ [f,g]_D\big),\end{equation*}  \\

\noindent for $(X,f),(Y,g)\in H(\fr{h})$, $X,Y\in \fr{h}$ and $f,g\in D(\fr{h})$.  The Lie algebra $\fr{h}$ can be embedded as an ideal of its holomorph via the isomorphism $\fr{h}\rightarrow \fr{h} \rtimes \{0\}$, with $X\mapsto (X,0)$, $X\in \fr{h}$. We will need the following.

\begin{lemma}\label{characteristic}Let $\fr{k}$ be a Lie algebra and let $\fr{h}$ be an ideal of $\fr{k}$ satisfying condition (ii) of Theorem \ref{mainperfect2}.  Then $\fr{h}$ is a characteristic ideal of $\fr{k}$, i.e. any derivation of $\fr{k}$ leaves $\fr{h}$ invariant.\end{lemma}

\begin{proof} Let $f\in D(\fr{k})$ and $X\in \fr{h}$. To verify that $f(X)\in \fr{h}$, consider the subideal sequence $\fr{h} \rtimes \{0\}\unlhd \fr{k} \rtimes \{0\}\unlhd H(\fr{k})$.  Our hypothesis implies that $\fr{h} \rtimes \{0\}\unlhd H(\fr{k})$.  Therefore, $\big[(X,0),(0,f)\big]_H\in \fr{h} \rtimes \{0\}$, which yields $(-f(X),0)\in \fr{h} \rtimes \{0\}$, i.e. $f(X)\in \fr{h}$.    \end{proof}


We are now ready to prove Theorem \ref{mainperfect2}. \\

\noindent \emph{Proof of Theorem \ref{mainperfect2}.}  The theorem can be immediately verified if $\fr{h}$ is trivial.  Suppose that $\fr{h}$ is not trivial, and assume firstly that $\fr{h}$ is perfect. To prove that $\fr{h}$ satisfies condition (ii) of the theorem, assume that $\fr{h}\unlhd \fr{k}\unlhd \fr{g}$ for some Lie algebras $\fr{k}$ and $\fr{g}$. Then $[\fr{h},\fr{g}]\subseteq \fr{k}$ and $[\fr{h}, \fr{k}]\subseteq \fr{h}$.  Since $\fr{h}=[\fr{h},\fr{h}]$, the Jacobi identity yields 

\begin{equation*}\label{Jacper}[\fr{h},\fr{g}]=\big[[\fr{h},\fr{h}],\fr{g}\big]\subseteq \big[\fr{h},[\fr{h},\fr{g}]\big]\subseteq [\fr{h},\fr{k}]\subseteq \fr{h}.\end{equation*}

\noindent Therefore, $\fr{h}$ is an ideal of $\fr{g}$.  Conversely, assume that $\fr{h}$ satisfies condition (ii) of the theorem and suppose on the contrary that $\fr{h}$ is not perfect.  Then the Lie algebra $\fr{h}/[\fr{h},\fr{h}]$ is non-trivial and abelian.  Consider the Lie algebra  

\[ \fr{k}:=\fr{h}\oplus \big(\fr{h}/[\fr{h},\fr{h}]\big),\]  
 
\noindent the canonical projection $\pi:\fr{h}\rightarrow \fr{h}/[\fr{h},\fr{h}]$ and the linear endomorphism $f:\fr{k}\rightarrow \fr{k}$ with 

\begin{equation*}f\big((X,Y)\big):=(0,\pi(X)),\end{equation*}

\noindent for $X\in \fr{h}$, $Y\in \fr{h}/[\fr{h},\fr{h}]$.  Denote by $[ \ , \ ]_{\fr{k}}$ the induced Lie bracket on $\fr{k}$.  Then for all $(X,Y), (Z,W)\in \fr{k}$, with $X,Z\in \fr{h}$, and $Y,W\in \fr{h}/[\fr{h},\fr{h}]$, we have on the one hand   

\begin{equation}\label{LH}f\big(\big[(X,Y),(Z,W)\big]_{\fr{k}}\big)=f\big(([X,Z], [Y,W])\big)=\big(0,\pi([X,Z])\big)=0,\end{equation}

\noindent given that $\pi([\fr{h},\fr{h}])=\{0\}$.  On the other hand, 

\begin{equation}\label{RH}\big[f\big((X,Y)\big),(Z,W)\big]_{\fr{k}}+\big[(X,Y),f\big((Z,W)\big)\big]_{\fr{k}}=\big(0, [\pi(X),W]+[Y,\pi(Z)]\big)=0,\end{equation}

\noindent given that the Lie algebra $\fr{h}/[\fr{h},\fr{h}]$ is abelian.  Equations \eqref{LH} and \eqref{RH} imply that $f$ is a derivation of $\fr{k}$. Since $\fr{h}=\fr{h}\oplus \{0\}$ is an ideal of $\fr{k}$ that satisfies satisfies condition (ii) of Theorem \ref{mainperfect2}, Lemma \ref{characteristic} implies that 

\begin{equation}\label{contrad}f(\fr{h}\oplus \{0\})\subseteq \fr{h}\oplus \{0\}.\end{equation}

\noindent But since $\fr{h}$ is not perfect, we can choose a $X_o\in \fr{h}$ such that $\pi(X_o)\neq 0$, and thus the definition of $f$ yields $f\big((X_o,0)\big)=(0,\pi(X_o))\notin  \fr{h}\oplus \{0\}$, contradicting relation \eqref{contrad}. We conclude that $\fr{h}$ is perfect. \qed \\

We remark that by combining Lemma \ref{characteristic} with the above proof of Theorem \ref{mainperfect2}, we arrive to the following equivalent characterization of perfect Lie algebras.

\begin{corol}\label{mainperfect3} Let $\fr{h}$ be a Lie algebra over a field $\mathbb K$.  The following are equivalent:\\
\noindent \emph{(i)} The Lie algebra $\fr{h}$ is perfect.\\ 
\noindent \emph{(ii)} For all Lie algebras $\fr{g}$, if $\fr{h}$ is an ideal of $\fr{g}$ then $\fr{h}$ is a characteristic ideal of $\fr{g}$. \end{corol}

Now assume that a Lie algebra $\fr{g}$ has trivial center.  Then all derivation algebras, defined inductively by $D^1(\fr{g}):=D(\fr{g})$, $D^{n+1}(\fr{g}):=D(D^n(\fr{g}))$, $n\in \mathbb N$, also have trivial center (see for example \cite{Sh} or \cite{SuZhu}).  Therefore, we can embed $D^n(\fr{g})$ as an ideal of $D^{n+1}(\fr{g})$ via the faithfull adjoint representation $\op{ad}:D^n(\fr{g}) \rightarrow D^{n+1}(\fr{g})$, $X\mapsto \operatorname{ad}_X$, where the algebra $D^n(\fr{g})$ is identified with the Lie algebra $\operatorname{ad}_{D^n(\fr{g})}$ of its inner derivations.  Thus one can define a sequence of ideals, the so-called \emph{derivation tower} of $\fr{g}$,  

\[ \fr{g}\unlhd D(\fr{g})\unlhd D^2(\fr{g})\unlhd \cdots \unlhd D^n(\fr{g})\unlhd \cdots .\]

\noindent A main result in \cite{Sh} shows that for a finite-dimensional Lie algebra $\fr{g}$ with trivial center, the above sequence stabilizes after a finite number of steps, i.e. there exists an $n\in \mathbb N$ such that $D^{n+1}(\fr{g})$ is isomorphic to $D^n(\fr{g})$. 

 A Lie algebra $\fr{g}$ is called \emph{complete} if it has trivial center and all its derivations are inner, that is $D(\fr{g})=\operatorname{ad}_{\fr{g}}$. Some results on complete Lie algebras can be found in \cite{Hsie}, \cite{JiMe} and the references therein. In \cite{SuZhu}, it is proved that the derivation tower of a perfect Lie algebra $\fr{g}$ with trivial center stabilizes at $n=1$, irrespectively of the dimension.  In other words, the derivation algebra $D(\fr{g})$ is complete.  

\begin{theorem}\label{SuZhu} \emph{(\cite{SuZhu})} Let $\fr{g}$ be a perfect Lie algebra with trivial center.  Then the derivation algebra $D(\fr{g})$ is complete, i.e. $D^2(\fr{g})=D(\fr{g})$. \end{theorem}

 We will show that Theorem \ref{SuZhu} can be recovered by the transitivity property of perfect Lie algebras in Theorem \ref{mainperfect2}. To this end, let us formulate a generalization of Theorem \ref{SuZhu}.

\begin{theorem}\label{derived}Let $\fr{g}$ be a Lie algebra with trivial center. Then $D(\fr{g})$ is complete if and only if $\fr{g}$ is an ideal of $D^2(\fr{g})$.  \end{theorem}

\begin{remark}\label{SuShu}Theorem \ref{SuZhu} follows immediately as a corollary of Theorem \ref{derived}.  Indeed, if $\fr{g}$ is a perfect Lie algebra with trivial center then by Theorem \ref{mainperfect2}, the subideal sequence $\fr{g}\unlhd D(\fr{g})\unlhd D^2(\fr{g})$ yields $\fr{g}\unlhd D^2(\fr{g})$, and thus $D(\fr{g})$ is complete by Theorem \ref{derived}.\end{remark}

 To prove Theorem \ref{derived}, we need the following lemma which can be deduced by intermediate arguments in \cite{SuZhu}.  We provide a proof here for completeness and to establish some notation.

\begin{lemma}\label{lem} Let $\big(\fr{g}, [ \ , \ ]_{\fr{g}}\big)$ be a Lie algebra with trivial center, let $\big(D(\fr{g}), [ \ , \ ]_D\big)$ be the Lie algebra of derivations of $\fr{g}$ and let $D^2(\fr{g})$ be the Lie algebra of derivations of $D(\fr{g})$.  Denote by $\op{ad}^{\fr{g}}:\fr{g}\rightarrow D(\fr{g})$, $X\mapsto\op{ad}^{\fr{g}}_X$, the adjoint representation of $\fr{g}$.  Let $f\in D(\fr{g})$ and $F\in D^2(\fr{g})$.\\

\noindent \emph{(i)} For all $X\in \fr{g}$, it follows that $[f,\op{ad}^{\fr{g}}_X]_D=\op{ad}^{\fr{g}}_{f(X)}$.\\

 \noindent \emph{(ii)} If $F(\op{ad}^{\fr{g}}_X)=0$ for all $\op{ad}^{\fr{g}}_X\in D(\fr{g})$, then $F=0$.    \end{lemma}

\begin{proof} (see also \cite{SuZhu}) For (i), let $Y\in \fr{g}$.  Given that $f$ is a derivation of $\fr{g}$, we have

\[ [f,\op{ad}^{\fr{g}}_X]_{D}Y=f([X,Y]_{\fr{g}})-[X,f(Y)]_{\fr{g}}=[f(X),Y]_{\fr{g}}=\op{ad}^{\fr{g}}_{f(X)}Y. \]

\noindent Since $Y$ is arbitrary, the above equation yields part (i). 

For part (ii), let $g\in D(\fr{g})$ be arbitrary. It suffices to prove that $F(g)=0$. For arbitrary $Y\in \fr{g}$, given that $F$ is a derivation of $D(\fr{g})$, using part (i) and the hypothesis that $F(\op{ad}^{\fr{g}}_X)=0$ for all $\op{ad}^{\fr{g}}_X\in D(\fr{g})$, we have

\begin{eqnarray*}\op{ad}^{\fr{g}}_{F(g)(Y)}&=&[F(g),\op{ad}^{\fr{g}}_Y]_{D}=F\big( [g,\op{ad}^{\fr{g}}_Y]_{D}\big)-[g,F(\op{ad}^{\fr{g}}_Y)]_{D}\\
&=&F\big( [g,\op{ad}^{\fr{g}}_Y]_{D}\big)=F(\op{ad}^{\fr{g}}_{g(Y)})=0.\end{eqnarray*}

\noindent Since $\fr{g}$ has trivial center, $\op{ad}^{\fr{g}}$ is faithful and thus $F(g)(Y)=0$ for all $Y\in \fr{g}$, i.e. $F(g)=0$.\end{proof}

\emph{Proof of Theorem \ref{derived}}.  Let $\op{ad}^{\fr{g}}:\fr{g}\rightarrow D(\fr{g})$ and $\op{ad}^{D}:D(\fr{g})\rightarrow D^2(\fr{g})$ denote the adjoint representations of $\fr{g}$ and $D(\fr{g})$ respectively. If $D(\fr{g})$ is complete then $D^2(\fr{g})$ is isomorphic to $\op{ad}^D_{D(\fr{g})}$, which is isomorphic to $D(\fr{g})$.  Since $\fr{g}$ is an ideal of $D(\fr{g})$, it is also an ideal of $D^2(\fr{g})$.  Conversely, assume that $\fr{g}$ is an ideal of $D^2(\fr{g})$ and consider the subideal sequence $\fr{g}\unlhd D(\fr{g})\unlhd D^2(\fr{g})$.  We will prove that $D^2(\fr{g})$ is isomorphic to $D(\fr{g})$.  Denote by $[ \ , \ ]_{D^2}$ the Lie bracket in $D^2(\fr{g})$. Identifying the algebras $\fr{g}$ and $D(\fr{g})$ with their corresponding embedded images in $D^2(\fr{g})$, we may write 

\begin{equation}\label{dd} D(\fr{g})=\op{ad}^D_{D(\fr{g})} \ \makebox{and} \ \fr{g}=\op{ad}^{D}_{\op{ad}^{\fr{g}}_{\fr{g}}}. \end{equation}

\noindent Since the adjoint representations $\op{ad}:D^i(\fr{g})\rightarrow D^{i+1}(\fr{g})$ are faithfull (given that $\fr{g}$ and all its derivation algebras have trivial center), any element $F$ of $D^2(\fr{g})$ defines a unique derivation of $D(\fr{g})$ by

\begin{equation*}\label{ddo}F\cdot \op{ad}^{D}_{f}:=[F,\op{ad}^{D}_{f}]_{D^2}=\op{ad}^{D}_{F(f)},\end{equation*}

\noindent for all  $\op{ad}^{D}_{f}\in D(\fr{g})$, as follows from Lemma \ref{lem}. Restricting the above map to $D(\fr{g})=\op{ad}^D_{D(\fr{g})}\subseteq D^2(\fr{g})$, any element $F=\op{ad}^D_f$ of $D(\fr{g})$ defines a unique derivation of $\fr{g}=\op{ad}^{D}_{\op{ad}^{\fr{g}}_{\fr{g}}}$ by 

\begin{equation*}\label{ddo}F\cdot \op{ad}^D_{\op{ad}^{\fr{g}}_X}=[\op{ad}^D_f,\op{ad}^D_{\op{ad}^{\fr{g}}_X}]_{D^2}=\op{ad}^{D}_{[f,\op{ad}^{\fr{g}}_X]_D}=\op{ad}^D_{\op{ad}^{\fr{g}}_{f(X)}},\end{equation*}

\noindent for all $\op{ad}^D_{\op{ad}^{\fr{g}}_X}\in \fr{g}$.  Since $\fr{g}\unlhd D^2(\fr{g})$, it follows that $[D^2(\fr{g}),\fr{g}]_{D^2}\subseteq \fr{g}$, and thus we have a Lie algebra homomorphism $\phi: D^2(\fr{g})\rightarrow D(\fr{g})$ with 

\begin{equation*}\label{dam} \phi(F)(\tilde{X}):=F\cdot \tilde{X}=[F,\tilde{X}]_{D^2},\end{equation*}

\noindent for $F\in D^2(\fr{g})$ and $\tilde{X}\in \fr{g}$. We will prove that $\phi$ is an isomorphism, which will conclude the proof of the proposition. The homomorphism $\phi$ is surjective because for all $F\in D(\fr{g})$ and $\tilde{X}\in \fr{g}$ we have $\phi(F)(\tilde{X})=F\cdot \tilde{X}$, and hence $\phi(F)=F$ as unique derivations of $\fr{g}$.  To prove that $\phi$ is injective, let $F\in \op{Ker}(\phi)$.  Then by virtue of part (i) of Lemma \ref{lem}, and in view of the identification \eqref{dd}, for all $\tilde{X}=\op{ad}^{D}_{\op{ad}^{\fr{g}}_{X}}\in \fr{g}$ we have

\[ 0=\phi(F)(\tilde{X})=[F,\tilde{X}]_{D^2}=[F,\op{ad}^{D}_{\op{ad}^{\fr{g}}_{X}}]_{D^2}=\op{ad}^{D}_{F(\op{ad}^{\fr{g}}_{X})}.\]

\noindent Since $\op{ad}^D$ is faithful, the above equation implies that $F(\op{ad}^{\fr{g}}_{X})=0$ for all $\op{ad}^{\fr{g}}_X\in D(\fr{g})$, and thus part (ii) of Lemma \ref{lem} yields $F=0$.  This concludes that $\phi$ is injective, and hence $D^2(\fr{g})$ is isomorphic to $D(\fr{g})$.  \qed
\\

In the sequel, we address the question of transitivity for complete subideals. Since there exist complete but non-perfect Lie algebras (e.g. any complete solvable Lie algebra \cite{MenZ}), Theorem \ref{mainperfect2} implies that not all complete Lie algebras satisfy the transitivity property (ii) of its statement. However, we have the following weaker property.

\begin{prop}\label{complete}Let $\fr{h}$ be a complete Lie algebra over a field $\mathbb K$.  Assume that there exist Lie algebras $\fr{k}$ and $\fr{g}$ such that $\fr{h}\unlhd\fr{k}\unlhd \fr{g}$.  If $\fr{k}$ has trivial center, then $\fr{h}$ is an ideal of $\fr{g}$.\end{prop}  

\begin{proof} Let $\op{ad}^{\fr{h}}:\fr{h}\rightarrow D(\fr{h})$ and $\op{ad}^{\fr{k}}:\fr{k}\rightarrow D(\fr{k})$ denote the adjoint representations of $\fr{h}$ and $\fr{k}$ respectively.  Given that $\fr{h}$ is an ideal of $\fr{k}$, we can consider the homomorphism $\phi:\fr{k}\rightarrow D(\fr{h})$, with $\phi(X):=\left.\operatorname{ad}^{\fr{k}}_X\right|_{\fr{h}}:\fr{h}\rightarrow \fr{h}$, $X\in \fr{k}$. Since $\fr{h}$ is complete, we have the isomorphism $\fr{h}=D(\fr{h})=\op{ad}^{\fr{h}}_{\fr{h}}=\left.\op{ad}^{\fr{k}}_{\fr{h}}\right|_{\fr{h}}$, and thus the homomorphism $\phi$ is surjective.  Moreover, the kernel of $\phi$ is the centralizer of $\fr{h}$ in $\fr{k}$, i.e. the ideal of $\fr{k}$ given by

\[ \fr{c}_{\fr{k}}(\fr{h})=\{X\in \fr{k}:\operatorname{ad}^{\fr{k}}_XY=0 \ \makebox{for all} \ Y\in \fr{h}\}.\]

\noindent Therefore, we have a canonical isomorphism $\psi:\fr{k}/\fr{c}_{\fr{k}}(\fr{h})\rightarrow D(\fr{h})=\fr{h}$ and a short exact sequence 

\[ 0 \xrightarrow[\text{}]{}  \fr{c}_{\fr{k}}(\fr{h})\xrightarrow[\text{}]{\text{$i$}} \fr{k}\xrightarrow[\text{}]{\text{$\hat{\pi}$}}\fr{h} \xrightarrow[]{} 0,   \]

\noindent where $i$ is the inclusion map and $\hat{\pi}=\psi\circ \pi$, with $\pi:\fr{k}\rightarrow \fr{k}/\fr{c}_{\fr{k}}(\fr{h})$ being the canonical projection.  This yields the vector space direct sum

\begin{equation}\label{dirsum}\fr{k}=\fr{h}\oplus \fr{c}_{\fr{k}}(\fr{h}),\end{equation}

\noindent which is also a direct sum of Lie algebras, given that $[\fr{h},\fr{c}_{\fr{k}}(\fr{h})]=\{0\}$, as follows from the definition of $\fr{c}_{\fr{k}}(\fr{h})$.  Therefore, the orthogonal projection $\pi_{\fr{c}}:\fr{k}\rightarrow \fr{c}_{\fr{k}}(\fr{h})$ with respect to the decomposition \eqref{dirsum} is a homomorphism of Lie algebras. 

To prove that $\fr{h}$ is an ideal of $\fr{g}$, let $X\in \fr{g}$, $Y\in \fr{h}$ and $Z\in \fr{c}_{\fr{k}}(\fr{h})$.  Since $\fr{k}\unlhd \fr{g}$, we have 

\begin{equation}\label{inclus}[X,Y]\in \fr{k}.\end{equation}

\noindent On the other hand, by taking into account the facts that $[Y,Z]=0$, $\fr{k}\unlhd \fr{g}$ and $\fr{h}\unlhd \fr{k}$, we obtain

\begin{eqnarray*}\big[\pi_{\fr{c}}([X,Y]),Z\big]&=&\big[\pi_{\fr{c}}([X,Y]),\pi_{\fr{c}}(Z)\big]=\pi_{\fr{c}}\big(\big[[X,Y],Z\big]\big)\\
&=&-\pi_{\fr{c}}\big(\big[[Y,Z],X\big]\big)-\pi_{\fr{c}}\big(\big[[Z,X],Y\big]\big)\\
&=&
-\pi_{\fr{c}}\big(\big[[Z,X],Y\big]\big)\in \pi_{\fr{c}}([\fr{k},\fr{h}])\subseteq \pi_{\fr{c}}(\fr{h})=\{0\}.
 \end{eqnarray*}

\noindent Therefore, $\pi_{\fr{c}}([X,Y])$ lies in the center of $\fr{c}_{\fr{k}}(\fr{h})$.  But since $\fr{k}$ has trivial center, the direct sum \eqref{dirsum} implies that $\fr{c}_{\fr{k}}(\fr{h})$ also has trivial center.  We conclude that $\pi_{\fr{c}}([X,Y])=0$ which, in view of relation \eqref{inclus} and the direct sum \eqref{dirsum}, implies that $[X,Y]$ in $\fr{h}$ for all $X\in \fr{g}$ and $Y\in \fr{h}$. \end{proof}

\section{Some extrinsic conditions for subideals}\label{SecCom}

In this section, we consider conditions that rely on the embedding of a Lie algebra $\fr{h}$ in a Lie algebra $\fr{g}$, such that if $\fr{h}$ is a subideal of $\fr{g}$ then $\fr{h}$ is an ideal of $\fr{g}$.  All Lie algebras in this section are finite-dimensional over a field of characteristic zero. Firstly, we recall the following.

\begin{theorem}\label{radicalth}\emph{(\cite{Sh}, Th. 6)} Let $\fr{g}$ be a finite-dimensional Lie algebra over a field of characteristic zero and let $\fr{h}$ be a subideal of $\fr{g}$.  Denote by $\fr{r}_{\fr{g}}$, $\fr{r}_{\fr{h}}$ the radicals of $\fr{g}$, $\fr{h}$ respectively.  Then $\fr{r}_{\fr{h}}=\fr{r}_{\fr{g}}\cap \fr{h}$.\end{theorem}

As a consequence, we obtain the following criterion for subideals in characteristic zero.

\begin{prop}\label{Leviconsequence}Let $\fr{g}$ be a finite-dimensional Lie algebra over a field of characteristic zero and let $\fr{h}$ be a subideal of $\fr{g}$ with radical $\fr{r}_{\fr{h}}$.  The following are equivalent:\\

\noindent \emph{(i)} The Lie algebra $\fr{h}$ is an ideal of $\fr{g}$.\\
\noindent \emph{(ii)} The radical $\fr{r}_{\fr{h}}$ of $\fr{h}$ is an ideal of $\fr{g}$.\\
\noindent \emph{(iii)} The radical $\fr{r}_{\fr{h}}$ of $\fr{h}$ satisfies the relation $[\fr{r}_{\fr{h}}, \fr{g}]\subseteq \fr{h}$.\end{prop}

\begin{proof} By Theorem \ref{radicalth}, we have $\fr{r}_{\fr{h}}=\fr{r}_{\fr{g}}\cap \fr{h}$. If $\fr{h}$ is an ideal of $\fr{g}$ then $\fr{r}_{\fr{h}}$ is the intersection of ideals of $\fr{g}$, and hence it is an ideal of $\fr{g}$. Therefore, (i) implies (ii). The fact that (ii) implies (iii) is trivial. Finally, we will prove that (iii) implies (i). Since $\fr{h}$ is a subideal of $\fr{g}$, there exist Lie subalgebras $\fr{l}_0,\fr{l}_1,\fr{l}_2\dots,\fr{l}_n$ of $\fr{g}$ such that 

\begin{equation}\label{ttempp}\fr{h}=\fr{l}_0\unlhd \fr{l}_1\unlhd \fr{l}_2 \unlhd\cdots \unlhd \fr{l}_n=\fr{g}.\end{equation}

\noindent  Using the Levi decomposition, we can write $\fr{h}$ as the semidirect sum 

\begin{equation}\label{semiLevi}\fr{h}=\fr{h}_s\rtimes \fr{r}_{\fr{h}},\end{equation}

\noindent where $\fr{h}_s$ is a semisimple Lie algebra, and thus $\fr{h}_s=[\fr{h}_s,\fr{h}_s]$.  Along with the fact that $\fr{h}\unlhd \fr{l}_1 \unlhd \fr{l}_2$, the Jacobi identity yields

\begin{equation*} [\fr{h}_s, \fr{l}_2]=[[\fr{h}_s,\fr{h}_s],\fr{l}_2]\subseteq [\fr{h}_s,[\fr{h}_s, \fr{l}_2]]\subseteq [\fr{h},[\fr{h}, \fr{l}_2]]\subseteq [\fr{h},\fr{l}_1]\subseteq \fr{h}. \end{equation*} 

\noindent Combined with the fact that $[\fr{r}_{\fr{h}}, \fr{g}]\subseteq \fr{h}$ and in view of the Levi decomposition \eqref{semiLevi}, the above equation implies that $[\fr{h}, \fr{l}_2]\subseteq \fr{h}$ and thus $\fr{h}\unlhd \fr{l}_2$. Proceeding similarly by induction in relation \eqref{ttempp}, we conclude that $\fr{h}\unlhd \fr{g}$. \end{proof} 

\begin{corol}\label{newell} Let $\fr{g}$ be a finite-dimensional Lie algebra over a field of characteristic zero and let $\fr{h}$ be a Lie subalgebra of $\fr{g}$.  Assume that one of the following is true:\\

\noindent \emph{(i)} The radical $\fr{r}_{\fr{g}}$ of $\fr{g}$ is contained in $\fr{h}$.\\
\noindent \emph{(ii)} The Lie algebra $\fr{g}$ is reductive, i.e. $\fr{g}=\fr{g}_s\oplus \fr{z}(\fr{g})$ where $\fr{g}_s$ is a semisimple Lie algebra and $\fr{z}(\fr{g})$ is the center of $\fr{g}$.\\

\noindent Then $\fr{h}$ is a subideal of $\fr{g}$ if and only if $\fr{h}$ is an ideal of $\fr{g}$.\end{corol}

\begin{proof} If $\fr{h}$ is an ideal of $\fr{g}$ then it is trivially a subideal of $\fr{g}$.  Conversely, assume that $\fr{h}$ is a subideal of $\fr{g}$.  For (i), if $\fr{r}_{\fr{g}}$ is contained in $\fr{h}$, Theorem \ref{radicalth} yields $\fr{r}_{\fr{h}}=\fr{r}_{\fr{g}}$.  Since $\fr{r}_{\fr{g}} \unlhd \fr{g}$, Proposition \ref{Leviconsequence} implies that $\fr{h}\unlhd \fr{g}$.  For (ii), if $\fr{g}$ is reductive then $\fr{r}_{\fr{g}}=\fr{z}(\fr{g})$. Therefore $\fr{r}_{\fr{h}}=\fr{h}\cap \fr{r}_{\fr{g}}\subseteq \fr{z}(\fr{g})$ and thus $\fr{r}_{\fr{h}}$ is an ideal of $\fr{g}$. We conclude from Proposition \ref{Leviconsequence} that $\fr{h}$ is an ideal of $\fr{g}$.\end{proof}

We remark that under assumption (ii), Corollary \ref{newell} also follows from the results in \cite{Ge-Mu} and \cite{Va}.

\begin{corol}\label{radicalcentral} Let $\fr{g}$ be a finite-dimensional Lie algebra over a field of characteristic zero and let $\fr{h}$ be a Lie subalgebra of $\fr{g}$.  Assume that any of the following is true:\\

\noindent \emph{(i)} The radical $\fr{r}_{\fr{h}}$ of $\fr{h}$ is contained in the center of $\fr{g}$.\\
\emph{(ii)} The Lie algebra $\fr{h}$ is semisimple.\\

 \noindent Then for any Lie subalgebra $\fr{k}$ of $\fr{g}$, the algebra $\fr{h}$ is a subideal of $\fr{k}$ if and only if $\fr{h}$ is an ideal of $\fr{k}$.\end{corol}

\begin{proof} In both cases (i) and (ii) we have $[\fr{r}_{\fr{h}}, \fr{k}]=\{0\}$ (for case (ii), we have $\fr{r}_{\fr{h}}=\{0\}$), and the conclusion follows from Proposition \ref{Leviconsequence}. \end{proof}

\noindent Observe that under assumption (ii), the conclusion of Corollary \ref{radicalcentral} also follows from Theorem \ref{mainperfect2}, given that any finite-dimensional semisimple Lie algebra in characteristic zero is perfect.  

In the sequel, we set $\mathbb K=\mathbb R$ and we work with subalgebras $\fr{h}$ of $\fr{g}$ such that the operators $\op{ad}_X$, $X\in \fr{h}$, are skew-symmetric with respect to some bilinear and symmetric form of $\fr{g}$. A general picture is painted by the following.

\begin{theorem}\label{picture}Let $\fr{g}$ be a finite-dimensional real Lie algebra supplied with a symmetric bilinear form $B$, and let $\fr{h}$ be a Lie subalgebra of $\fr{g}$ such that the following hold:\\

\noindent \emph{(i)} The form $B$ is non-degenerate on $\fr{h}$ and positive definite on the $B$-orthogonal complement $\fr{m}$ of $\fr{h}$ in $\fr{g}$.\\

\noindent \emph{(ii)} Any operator $\op{ad}_X$, $X\in \fr{h}$, is skew-symmetric with respect to $B$, i.e. 

\[  B(\op{ad}_XY,Z)+B(Y,\op{ad}_XZ)=0, \ \makebox{for all} \ X\in \fr{h}, \ Y,Z\in \fr{g}.\]

\noindent Then for any Lie subalgebra $\fr{k}$ of $\fr{g}$, the algebra $\fr{h}$ is a subideal of $\fr{k}$ if and only if $\fr{h}$ is an ideal of $\fr{k}$.  \end{theorem}

\begin{proof} If $\fr{h}$ is an ideal of $\fr{k}$, then it is trivially a subideal of $\fr{k}$. To prove the converse statement, notice that since the subalgebra $\fr{k}$ of $\fr{g}$ is arbitrary, by induction it suffices to prove that for any subalgebras $\fr{k}_1$ and $\fr{k}_2$ of $\fr{g}$, the relation $\fr{h}\unlhd \fr{k}_1\unlhd \fr{k}_2$ implies that $\fr{h}\unlhd \fr{k}_2$.

To this end, let $\fr{k}_1$ and $\fr{k}_2$ be subalgebras of $\fr{g}$ such that $\fr{h}\unlhd \fr{k}_1\unlhd \fr{k}_2$.  Since $B$ is non-degenerate on $\fr{h}$, we can write $\fr{k}_1$ as the vector space direct sum

\[ \fr{k}_1=\fr{h}+\fr{m}_1, \]

\noindent where $\fr{m}_1$ is the $B$-orthogonal complement of $\fr{h}$ in $\fr{k}_1$.  Given that $\fr{m}$ is the $B$-orthogonal complement of $\fr{h}$ in $\fr{g}$, we have $\fr{m}_1\subseteq \fr{m}$. Since $B$ is positive definite on $\fr{m}$, it is also positive definite on $\fr{m}_1$, and thus the Lie algebra $\fr{h}$ is identified with the $B$-orthogonal complement of $\fr{m}_1$ in $\fr{k}_1$.  Moreover,  

\begin{equation}\label{eq1} [\fr{h}, \fr{m}_1]=\{0\}.\end{equation}

\noindent To see this, observe on the one hand that 

\begin{equation}\label{him}[\fr{h}, \fr{m}_1]\subseteq \fr{h},\end{equation}

\noindent as $\fr{h}$ is an ideal of $\fr{k}_1$. On the other hand, for all $X,Y \in \fr{h}$ and $Z\in \fr{m}_1$, the skew-symmetry condition on $B$ yields $B([X,Z], Y)=-B(Z,[X,Y])\in B(\fr{m}_1,\fr{h})=\{0\}$. Hence we also have $B([\fr{h}, \fr{m}_1],\fr{h})=\{0\}$, or equivalently $[\fr{h}, \fr{m}_1]\subseteq \fr{m}_1$ which, in view of relation \eqref{him} and the fact that $\fr{h}\cap \fr{m}_1=\{0\}$, yields relation \eqref{eq1}.

To show that $\fr{h}$ is an ideal of $\fr{k}_2$, notice that $[\fr{h},\fr{k}_2]\subseteq \fr{k}_1$, given that $\fr{k}_1\unlhd \fr{k}_2$. To conclude that $[\fr{h},\fr{k}_2]\subseteq \fr{h}$, it remains to show that $[\fr{h},\fr{k}_2]$ is $B$-orthogonal to $\fr{m}_1$.  To this end, by virtue of the skew-symmetry condition on $B$ and relation \eqref{eq1}, for all $X\in \fr{h}$, $Y\in \fr{k}_2$ and $Z\in \fr{m}_1$, we have 

\[ B([X,Y],Z)=-B(Y,[X,Z])=-B(Y,0)=0,\]

\noindent which yields the desired result that $B([\fr{h},\fr{k}_2],\fr{m}_1)=\{0\}$.\end{proof}

We proceed to give some examples of subalgebras $\fr{h}$ satisfying the assumptions of Theorem \ref{picture}.

\subsection*{Subalgebras of compact Lie algebras} Let $\fr{g}$ be a compact Lie algebra in the sense that $\fr{g}$ is the (real) Lie algebra of a compact Lie group $G$. It is well-known that $\fr{g}$ admits an inner product $B$ such that any operator $\op{ad}_X$, $X\in \fr{g}$, is skew-symmetric with respect to $B$. Therefore, any subalgebra $\fr{h}$ of $\fr{g}$ trivially satisfies the assumptions of Theorem \ref{picture}. Hence, we have the following.

\begin{corol}\label{CorollaryCompact}Let $\fr{h}$ be a Lie subalgebra of a compact Lie algebra $\fr{g}$.  Then for any Lie subalgebra $\fr{k}$ of $\fr{g}$, the Lie algebra $\fr{h}$ is a subideal of $\fr{k}$ if and only if $\fr{h}$ is an ideal of $\fr{k}$. \end{corol}

\subsection*{Compactly embedded Lie subalgebras}  A Lie subalgebra $\fr{h}$ of a real Lie algebra $\fr{g}$ is said to be \emph{compactly embedded} in $\fr{g}$ if there exists an inner product $B$ on $\fr{g}$ with respect to which any operator $\op{ad}_X$, $X\in \fr{h}$, is skew-symmetric. Equivalently, the closure of the group $\exp(\op{ad}_{\fr{h}})$ is compact in the group $\op{Aut}(\fr{g})$ of automorphisms of $\fr{g}$ (\cite{HilHof}, \cite{Wu}). For example, if $G$ is a Lie group and $H$ is a compact Lie subgroup of $G$, then the Lie algebra $\fr{h}$ of $H$ is compactly embedded in the Lie algebra $\fr{g}$ of $G$.  This can be proven (c.f. \cite{Ch-Eb}, proof of Proposition 3.16, p. 61) by choosing any inner product $\langle \ ,\ \rangle$ on $\fr{g}$ and a right-invariant volume form $\omega$ on $\op{Ad}_H$, induced from a right-invariant metric, where $\op{Ad}:G\rightarrow \op{Aut}(\fr{g})$ denotes the adjoint representation of $G$. Then the form $B:\fr{g}\times \fr{g}\rightarrow \mathbb R$, given by 

\begin{equation*}\label{form}B( X, Y):=\int_{\op{Ad}_H}{\langle \op{Ad}_hX,\op{Ad}_hY\rangle}\ \omega(h),\quad X,Y\in \fr{g},\end{equation*}

\noindent is well-defined and constitutes an $\op{Ad}_H$-invariant inner product.  As a result, any operator $\op{ad}_X$, $X\in \fr{h}$, is skew-symmetric with respect to $B$.  Moreover, if $(G/H,g)$ is a \emph{Riemannian homogeneous space} then the isotropy subgroup $H$ is compact (\cite{Nom}), and thus the Lie algebra $\fr{h}$ of the isotropy subgroup $H$ is compactly embedded in the Lie algebra $\fr{g}$ of $G$.

It is straightforward to confirm that any compact Lie algebra is compactly embedded in itself. Moreover, if $\fr{h}$ is compactly embedded in $\fr{g}$, then it is compactly embedded in any Lie subalgebra of $\fr{g}$ containing $\fr{h}$. Finally, if $\fr{h}$ is compactly embedded in $\fr{g}$ then any Lie subalgebra of $\fr{h}$ is also compactly embedded in $\fr{g}$. In view of the above discussion, Theorem \ref{picture} yields the following.

\begin{corol}\label{CorollaryCompactly}Let $\fr{g}$ be a real Lie algebra and let $\fr{h}$ be a compactly embedded subalgebra of $\fr{g}$.  Then for any Lie subalgebras $\fr{h}^{\prime}$ and $\fr{g}^{\prime}$ of $\fr{h}$ and $\fr{g}$ respectively, the algebra $\fr{h}^{\prime}$ is a subideal of $\fr{g}^{\prime}$ if and only if $\fr{h}^{\prime}$ is an ideal of $\fr{g}^{\prime}$. In particular, any compactly embedded subideal is an ideal.\end{corol}

\begin{proof}If $\fr{h}$ is compactly embedded in $\fr{g}$ then $\fr{h}^{\prime}$ is compactly embedded in $\fr{g}^{\prime}$.  In other words, there exists an inner product $B$ in $\fr{g}^{\prime}$ with respect to which any operator $\op{ad}_X$, $X\in \fr{h}^{\prime}$, is skew-symmetric.  Thus condition (ii) of Theorem \ref{picture} is satisfied.  Condition (i) of Theorem \ref{picture} is also satisfied since $B$ is an inner product.   Therefore, $\fr{h}^{\prime}$ is a subideal of $\fr{g}^{\prime}$ if and only if $\fr{h}^{\prime}$ is an ideal of $\fr{g}^{\prime}$. \end{proof}

\subsection*{Subalgebras containing an eigenspace of a Cartan involution}  Let $\fr{g}$ be a real semisimple Lie algebra with Killing form $B$ and let $\theta:\fr{g}\rightarrow \fr{g}$ be a \emph{Cartan involution} of $\fr{g}$, i.e. an automorphism of $\fr{g}$ with $\theta^2=\op{id}$ and such that the symmetric bilinear form $\langle X,Y\rangle =-B(X,\theta(Y))$ is positive definite.  Then one obtains the Cartan decomposition $\fr{g}=\fr{u}+ \fr{p}$ with respect to $\langle \ ,\ \rangle$, where $\fr{u}$ is a compact Lie subalgebra of $\fr{g}$ that coincides with the $+1$-eigenspace of $\theta$, i.e. $\theta(X)=X$ for all $X\in \fr{u}$, and $\fr{p}$ is a vector space that coincides with the $-1$-eigenspace of $\theta$, i.e. $\theta(X)=-X$ for all $X\in \fr{p}$ (c.f. \cite{Helg} p. 185). We have the following.

\begin{corol}\label{CartanSemi}Let $\fr{g}$ be a real finite-dimensional semisimple Lie algebra and let $\theta$ be a Cartan involution of $\fr{g}$.  Let $\fr{u}$ be the $+1$-eigenspace of $\theta$ and let $\fr{p}$ be the $-1$-eigenspace of $\theta$.   Assume that $\fr{h}$ is a Lie subalgebra of $\fr{g}$ containing either $\fr{u}$ or $\fr{p}$. Then for any Lie subalgebra $\fr{k}$ of $\fr{g}$, the algebra $\fr{h}$ is a subideal of $\fr{k}$ if and only if $\fr{h}$ is an ideal of $\fr{k}$. \end{corol}

\begin{proof}  Denote by $B$ the Killing form of $\fr{g}$ and let $\fr{g}=\fr{u} +\fr{p}$ be the Cartan decomposition with respect to the inner product $\langle X,Y\rangle=-B(X,\theta(Y))$.  Since any operator $\op{ad}_X$, $X\in \fr{g}$, is skew symmetric with respect to the Killing form $B$, condition (ii) of Theorem \ref{picture} is automatically satisfied.  The spaces $\fr{u}$ and $\fr{p}$ are $B$-orthogonal since for all $X\in \fr{u}$, $Y\in \fr{p}$ we have $B(X,Y)=-B(X,\theta(Y))=\langle X,Y\rangle =0$.  Moreover, since $\langle \ ,\  \rangle$ is positive definite, it follows that $B(X,X)=B(X,\theta(X))=-\langle X,X \rangle <0$ for all $X\in \fr{u}$ and $B(Y,Y)=-B(Y,\theta(Y))=\langle Y,Y \rangle >0$ for all $Y\in \fr{p}$.  In other words, $B$ is negative definite on $\fr{u}$ and positive definite on $\fr{p}$. Therefore, each of these spaces is the $B$-orthogonal complement of the other in $\fr{g}$.

Assume initially that $\fr{u}\subseteq \fr{h}$.  Since $B$ is negative definite on $\fr{u}$ (and thus non-degenerate on $\fr{u}$), we have a $B$-orthogonal vector space direct sum $\fr{h}=\fr{u}+\fr{q}$ where $\fr{q}\subseteq \fr{p}$. Moreover, $B$ is non-degenerate on $\fr{h}$.  To see this, let $Z\in \fr{h}$ such that $B(Z, \fr{h})=\{0\}$; Writing $Z=Z_{\fr{u}}+Z_{\fr{q}}$ with respect to the direct sum $\fr{h}=\fr{u}+\fr{q}$, we obtain $B(Z, Z_{\fr{u}})=0$ and $B(Z, Z_{\fr{q}})=0$ which, in view of the $B$-orthogonality of $\fr{u},\fr{q}$, is equivalent to $B(Z_{\fr{u}}, Z_{\fr{u}})=0$ and $B(Z_{\fr{q}}, Z_{\fr{q}})=0$.  But since $B$ is negative definite on $\fr{u}$ and positive definite on $\fr{q}\subseteq \fr{p}$, the last equations yield $Z_{\fr{u}}=Z_{\fr{q}}=0$, i.e. $Z=0$.  Now given that $B$ is non-degenerate on $\fr{h}$, we can write $\fr{g}$ as the $B$-orthogonal vector space direct sum $\fr{g}=\fr{h}+\fr{m}$, where $\fr{m}$ is the $B$-orthogonal complement of $\fr{h}$ in $\fr{g}$.  Since $\fr{h}$ contains $\fr{u}$, and along with the fact that $\fr{p}$ coincides with the $B$-orthogonal complement of $\fr{u}$ in $\fr{g}$, it follows that $\fr{m}\subseteq \fr{p}$ and hence $B$ is positive definite on $\fr{m}$.  In summary, we have proven that $B$ is non-degenerate on $\fr{h}$ and positive definite on the $B$-orthogonal complement $\fr{m}$ of $\fr{h}$ in $\fr{g}$.  Therefore, $\fr{h}$ satisfies both conditions of Theorem \ref{picture}, concluding the proof for the case $\fr{u}\subseteq \fr{h}$.

Finally, assume that $\fr{p}\subseteq \fr{h}$.  The proof is analogous to the case $\fr{u}\subseteq \fr{h}$, so we describe it briefly.  We consider the form $(-B)$ and the $(-B)$-orthogonal vector space direct sum $\fr{h}=\fr{p}+\fr{q}^{\prime}$, where $\fr{q}^{\prime}\subseteq \fr{u}$. Again, $(-B)$ is non-degenerate on $\fr{h}$, and thus we have $\fr{g}=\fr{h}+\fr{m}^{\prime}$, where $\fr{m}^{\prime}$ is the $(-B)$-orthogonal complement of $\fr{h}$ in $\fr{g}$.  Since $\fr{h}$ contains $\fr{p}$, it follows that $\fr{m}^{\prime}\subseteq \fr{u}$ and hence $(-B)$ is positive definite on $\fr{m}^{\prime}$.  Therefore, $\fr{h}$ satisfies both conditions of Theorem \ref{picture}, concluding the proof for the case $\fr{p}\subseteq \fr{h}$.\end{proof}

\begin{remark}\label{simple}We note that if $\fr{g}$ is simple and $(\fr{g},\fr{u})$ is an irreducible symmetric pair (see e.g. \cite{Helg}) then $\fr{u}$ is a maximal subalgebra of $\fr{g}$, and hence if $\fr{h}$ contains $\fr{u}$ then $\fr{h}=\fr{u}$ or $\fr{h}=\fr{g}$. \end{remark}

\begin{remark} Assume that $\fr{h}$ is contained in the $+1$ eigenspace $\fr{u}$ of a Cartan involution.  Since the Lie subalgebra $\fr{u}$ is compactly embedded in $\fr{g}$, then $\fr{h}$ is also compactly embedded in $\fr{g}$.  Hence Corollary \ref{CorollaryCompactly} implies that for any Lie subalgebras $\fr{h}^{\prime}$ and $\fr{g}^{\prime}$ of $\fr{h}$ and $\fr{g}$ respectively, the algebra $\fr{h}^{\prime}$ is a subideal of $\fr{g}^{\prime}$ if and only if $\fr{h}^{\prime}$ is an ideal of $\fr{g}^{\prime}$.\end{remark}

\section{An application on self-normalizing subalgebras of Lie algebras}\label{GeodesicOrbit}

Let $\fr{h}$ be a Lie subalgebra of a Lie algebra $\fr{g}$.  The normalizer of $\fr{h}$ in $\fr{g}$ is the subalgebra 

\[ N_{\fr{g}}(\fr{h})=\{X\in \fr{g}:[X,Y]\in \fr{h}\ \makebox{for all} \ Y\in \fr{h}\}, \]

\noindent or equivalently, the maximal subalgebra of $\fr{g}$ containing $\fr{h}$ as an ideal. The Lie algebra $\fr{h}$ is called self-normalizing in $\fr{g}$ if $\fr{h}= N_{\fr{g}}(\fr{h})$.  Foremost examples of self-normalizing subalgebras include Cartan subalgebras, as well as subalgebras of maximal rank (such as Borel and parabolic subalgebras) in semisimple Lie algebras (\cite{Bou}). Further classes of self-normalizing subalgebras were studied via Kostant pairs in \cite{Si}.  The following application of our main results yields large classes of self-normalizing subalgebras.

\begin{theorem}\label{SelfNormalizingTh} Let $\fr{h}$ be a Lie subalgebra of a Lie algebra $\fr{g}$ over a field $\mathbb K$.  Assume that any of the following is true:\\

\noindent \emph{(i)} The algebra $\fr{h}$ is perfect.\\
\emph{(ii)} The algebra $\fr{g}$ is finite-dimensional, $\mathbb K$ has characteristic zero and the radical $\fr{r}_{\fr{h}}$ of $\fr{h}$ is contained in the center of $\fr{g}$.\\
\emph{(iii)} The algebras $\fr{g}$ and $\fr{h}$ satisfy the assumptions of Theorem \ref{picture}.  More specifically, 

\noindent \emph{(iv)} The algebra $\fr{g}$ is finite-dimensional and compact with $\mathbb K=\mathbb R$.\\
\emph{(v)} The algebra $\fr{g}$ is finite-dimensional, $\mathbb K=\mathbb R$ and $\fr{h}$ is compactly embedded in $\fr{g}$.\\
 \emph{(vi)} The algebra $\fr{g}$ is finite-dimensional and semisimple, $\mathbb K=\mathbb R$ and $\fr{h}$ contains an eigenspace of a Cartan involution of $\fr{g}$.\\

\noindent Then the normalizer $N_{\fr{g}}(\fr{h})$ is a self-normalizing subalgebra of $\fr{g}$.  \end{theorem}

\begin{proof}  Consider the subideal series $\fr{h}\unlhd  N_{\fr{g}}(\fr{h})\unlhd  N_{\fr{g}}\big(N_{\fr{g}}(\fr{h})\big)$.  In any of the cases (i)--(vi) respectively, Theorem \ref{mainperfect2}, Corollary \ref{radicalcentral}, Theorem \ref{picture}, Corollary \ref{CorollaryCompact}, Corollary \ref{CorollaryCompactly} and Corollary \ref{CartanSemi} respectively yield $\fr{h}\unlhd N_{\fr{g}}\big(N_{\fr{g}}(\fr{h})\big)$.  Since $N_{\fr{g}}(\fr{h})$ is the maximal subalgebra of $\fr{g}$ containing $\fr{h}$ as an ideal, it follows that $N_{\fr{g}}(\fr{h})=N_{\fr{g}}\big(N_{\fr{g}}(\fr{h})\big)$, i.e. $N_{\fr{g}}(\fr{h})$ is a self-normalizing subalgebra of $\fr{g}$.\end{proof}

\end{document}